\documentclass[12pt]{amsart}

\numberwithin{equation}{section}

\usepackage{amsmath,amsthm,amsfonts,eucal,}
\usepackage{xypic}
\usepackage{graphicx}
\usepackage{amssymb}
\usepackage{mathtools}
\usepackage{url}

\def\cb{{\mathcal B}}

\def\ch{{\mathcal H}}

\def\ga{{\mathfrak A}}

\def\bc{{\mathbb C}}

\def\bj{{\mathbb J}}

\def\bn{{\mathbb N}}

\def\br{{\mathbb R}}
\def\bt{{\mathbb T}}

\def\bz{{\mathbb Z}}

\def\a{\alpha}

\def\g{\gamma}

\def\f{\varphi}  
  
\def\om{\omega}

\newtheorem{thm}{Theorem}[section]

\newtheorem{lem}[thm]{Lemma}
\newtheorem{cor}[thm]{Corollary}
\newtheorem{prop}[thm]{Proposition}

\theoremstyle{definition}

\begin{document}
\title[ de Finetti theorem on the infinite non-commutative torus]
{ de Finetti theorem on the infinite non-commutative torus}

\author{Vitonofrio Crismale}
\address{Vitonofrio Crismale\\
Dipartimento di Matematica\\
Universit\`{a} degli studi di Bari\\
Via E. Orabona, 4, 70125 Bari, Italy}
\email{\texttt{vitonofrio.crismale@uniba.it}}

\author{Simone Del Vecchio}
\address{Simone Del Vecchio\\
Dipartimento di Matematica\\
Universit\`{a} degli studi di Bari\\
Via E. Orabona, 4, 70125 Bari, Italy}
\email{\texttt{simone.delvecchio@uniba.it}}

\author{Maria Elena Griseta}
\address{Maria Elena Griseta\\
Dipartimento di Matematica\\
Universit\`{a} degli studi di Bari\\
Via E. Orabona, 4, 70125 Bari, Italy}
\email{\texttt{mariaelena.griseta@uniba.it}}

\author{Stefano Rossi}
\address{Stefano Rossi\\
Dipartimento di Matematica\\
Universit\`{a} degli studi di Bari\\
Via E. Orabona, 4, 70125 Bari, Italy}
\email{\texttt{stefano.rossi@uniba.it}}

\begin{abstract}

The set of spreadable states on an infinite non-commutive torus $\mathbb{A}_\a^\bz$ is 
determined for all values of the deformation parameter $\a$. If $\frac{\a}{2\pi}$ is irrational, the canonical
trace is the only  spreadable state. If $\frac{\a}{2\pi}$ is rational, the set of all spreadable states is a Bauer simplex.
Moreover, its boundary is the set of all infinite products  of a single state
on $C(\bt)$, which is invariant under all rotations
by $n_0$-th roots of unity, where $n_0=p_1^{\{\frac{m_1}{2}\}}\cdots p_r^{\{\frac{m_r}{2}\}}$, with
$\frac{n}{2}$ being $\frac{n}{2}$ for $n$ even and $\frac{n-1}{2}$ for $n$ odd, if $\frac{q_1^{n_1}\ldots q_s^{n_s}}{p_1^{m_1}\ldots p_r^{m_r}}$ is the representation of $\frac{\a}{2\pi}$ in lowest terms.\\
Finally, the simplex of all stationary states on  $\mathbb{A}_\a^\bz$  is proved to be the Poulsen simplex for all
values of the deformation parameter $\a$.

\end{abstract}

\maketitle

\section{Introduction}

In Classical Probability, a sequence of random variable is exchangeable if its joint distribution
is invariant under all finite permutations. Indipendent, identically distributed sequences of random variables
are obvious instances of exchangeable sequences. The general picture is not too far from this particular example, in that (the joint distribution of)
an exchangeable sequence is a mixture of i.i.d. variables, by virtue of a very well-known result  essentially due to de Finetti,  see \cite{dF} for the original formulation of the theorem and \cite{HS} for later generalizations.
Arguably, it is more striking  that exchangeability is the same as spreadability, a seemingly
weaker symmetry where all subsquences of the given sequence have the same joint
distribution, as proved  by Ryll-Nardzewski, \cite{R}.\\
Distributional symmetries such as exchangeability and spreadability can also be studied at the more generel level of
non commutative $C^*$-dynamical systems. This can be done as soon as the state space of each single variable
is replaced with a (possibly non-commutative) $C^*$-algebra, often referred to as the sample algebra, the state space of the whole process with some infinite product ({\it e.g.} free,  graded tensor, tensor) of this $C^*$-algebra with itself, and the joint distribution with a state (a normalized positive linear functional) on the product $C^*$-algebra, see \cite{CriFid}. Notably, the infinite-product structure allows for a natural action of those symmetry groups naturally arising in Classical Probability.\\
To our knowledge, the first of these generalizations was considered in the seminal paper of  St\o rmer \cite{Sto},
where, among other things, it was shown that the set of exchangeable states on the (minimal) infinite tensor
product of a $C^*$-algebra with itself is a Choquet simplex, whose extreme states are precisely infinite products of a single state on the sample algebra, thus extending the analysis by Hewitt and Savage \cite{HS} for symmetric measures on an infinite Tykhonov product of a given compact state space .\\
The paper \cite{CFCMP} addresses  the same problem for the CAR algebra, whose grading comes into play and must be taken into account. Indeed, the set of exchangeable states 
 is again a Choquet simplex; however,  its extreme states are now infinite products of a single \emph{even} state on 
the sample algebra $\mathbb{M}_2$, the complex $2$ by $2$ matrices.\\
The spreadable states on the CAR algebra, instead, are the main focus of  \cite{CDRCAR}, where they are proved to be the same as exchangeable states.
Moreover, the CAR algebra is not only a natural setting where a Ryll-Nardzewski theorem is obtained, but also
a model where even rotatable states coincide with exchangeable states (a fully-fledged non-commutative version of Freedman's theorem does hold on the CCR algebra where invariance under rotations rules out non-gaussian extreme states,  \cite{CDMR}). Very recent work on de Finetti-type theorems in non-commutative settings can be
found in \cite{V}, where the author proves that the theorem holds for an infinite  tensor product twisted by the Klein group.\\
In this paper, we draw our attention to the states of the infinite non-commutative torus $\mathbb{A}_\a^\bz$, {\it i.e.} the universal $C^*$-algebra generated by a countable set of unitaries
$\{u_l:  l\in\bz\}$ satisfying the commutation relations $u_lu_k=e^{i2\pi\a}u_ku_l$ for any
$l, k\in\bz$ with $l<k$, where $\a$ is a real parameter.
Naively, elements of this $C^*$-algebra can be thought of as functions in the infinite
 coordinates $u_l$ which take their values in $\bt$ but do not commute. 
States on  $\mathbb{A}_\a^\bz$ can  be accordingly  interpreted as suitable quantum stochastic processes.\\
Because the commutation rules are affected by index exchange, 
there is no natural way in which permutations may be implemented as $*$-automorphism of $\mathbb{A}_\a^\bz$.  Even so,
strictly increasing maps of
$\bz$ into itself do act through endomorphisms by displacing the indices of the generators. This  makes it possible to consider
spreadability in our setting all the same. In that regard, the deformation parameter $\a$ plays a key  role. Indeed, if
$\frac{\a}{2\pi}$ is irrational, then the set of spreadable states is the singleton of the canonical trace of
$\mathbb{A}_\a^\bz$, Theorem \ref{spreadability}. When $\frac{\a}{2\pi}$ is rational, far more spreadable states exist. More precisely, in Theorem \ref{spreadability} we show they make up a Bauer simplex whose extreme states are infinite products  of a single state
on $C(\bt)$, the continuous functions on the one-dimensional torus $\bt$, that must be invariant under all rotations
by $n_0$-th roots of unity, where $n_0$ is a natural number associated with $\frac{\a}{2\pi}$.
The ultimate reason why this happens is that, for $\frac{\a}{2\pi}$ rational, $\mathbb{A}_\a^\bz$ contains
an abelian subalgebra isomorphic with $C(\bt^\bn)$, the continuous functions on the infinite Tychonov product
$\bt^\bn$. Phrased differently, our non-commutative model is as large as to host all classical $\bt$-valued sequences of random variables.\\
Finally, we lavish our attention on stationary states on $\mathbb{A}_\a^\bz$, namely those states invariant under the action of the shift $\tau$, $\tau(u_l)=u_{l+1}$, $l$ in $\bz$. Regardless of whether
$\frac{\a}{2\pi}$ is rational or not, the compact convex set of all stationary states is no longer a Bauer simplex.
Quite the opposite, it is a metrizable Choquet simplex with dense extreme set, as we show in  Theorem \ref{Poulsen}. Therefore, the set of stationary set is  affinely homeomorphic with the Poulsen simplex \cite{P}
thanks to a well-known deep result of Lindenstrauss {\it et al.} \cite[Theorem 2.3]{Lind} that, up to affine homeomorphisms, there is only one metrizable
Choquet simplex with dense extreme set.

\section{Preliminaries}

We denote by $\mathbb{A}_\a^\bz$ the infinite torus with deformation parameter $\a\in\br$,
that is the universal $C^*$-algebra generated by a countable set of unitaries
$\{u_l:  l\in\bz\}$ satisfying the commutation relations $u_lu_k=e^{i2\pi\a}u_ku_l$ for any
$l, k\in\bz$ with $l<k$.\\
By universality, associated with any strictly increasing map $h:\bz\rightarrow\bz$,  there is a $*$-endomorphism $\a_h$ acting on $\mathbb{A}_\a^\bz$ 
completely determined by 
\begin{equation}
\a_h(u_l)= u_{h(l)}\,, \textrm{for all}\,\, l\in \bz\,.
\end{equation}
Since $\a_{hg}=\a_{h}\circ\a_{g}$, for all increasing $h, g:\bz\rightarrow\bz$, we have that $\mathbb{A}_\a^\bz$ is acted upon by the semigroup 
of all strictly increasing maps.\\
For $\mathfrak{z}=(z_l)_{l\in\bz}\in\bt^\bz$, the infinite Tikhonov product of $\bt$ with itself, $\gamma_\mathfrak{z}$ is the automorphism of 
$\mathbb{A}_\a^\bz$ determined by

\begin{equation}\label{nondiagonalgauge}
\gamma_\mathfrak{z}(u_l)= z_lu_l\,, \quad\textrm{for all}\,\, l\in \bz\,.
\end{equation}

By averaging the action of $\bt^\bz$ (with respect to its Haar measure), we get the (faithful) canonical
trace ${\rm tr}$ of $\mathbb{A}_\a^\bz$, which is completely determined by

\begin{equation}
{\rm tr}(u_{l_1}^{k_1}\cdots u_{l_n}^{k_n})= \delta_{k_1, 0}\cdots\delta_{k_n, 0}\,, \quad n\in\bn, \,l_1, \ldots, l_n, k_1, \ldots k_n\in \bz\, .
\end{equation}
We next recall that $\bt$ acts as well on our torus through the so-called gauge automorphisms $\gamma_z$, determined by $\g_z(u_l):=zu_l$, for all $z$ in $\bt$. As a straightforward consequence of the gauge action, we see that, for every $l$ in $\bz$, the subalgebra $C^*(u_l)$ is canonically isomorphic with $C(\bt)$ (in the isomorphism that sends $u_l$ to the coordinate function $z$).
In the sequel, we will often identify $C^*(u_l)$ and $C(\bt)$ through this isomorphism.\\		
We denote by $E$ the conditional
expectation obtained by averaging the gauge action of $\bt$. Its action on monomials in the generators is 
\begin{equation}
E(u_{l_1}^{k_1}\cdots u_{l_n}^{k_n})=u_{l_1}^{k_1}\cdots u_{l_n}^{k_n}\delta_{k_1+\cdots+k_n, 0}, \quad n\in\bn, \,l_1, \ldots, l_n, k_1, \ldots k_n\in \bz\, .
\end{equation}
For $m$ in $\bz$, we can consider the $m$-th
 spectral eigenspace of $\mathbb{A}_\a^\bz$, defined
$$(\mathbb{A}_\a^\bz)^m:={\rm span}\{u_{l_1}^{k_1}\cdots u_{l_n}^{k_n}:  n\in\bn, l_1, \ldots, l_n\in\bz, k_1+\cdots+k_n=m\}\,.$$
If an element $a$ belongs to $(\mathbb{A}_\a^\bz)^m$ for some $m$ in $\bz$, we say that $a$ has degree $m$, and we simply write this as
$\partial a=m$.\\
Any two different spectral eigenspaces are clearly linearly independent, and the linear span
of all eigenspaces is a dense subalgebra of $\mathbb{A}_\a^\bz$.\\
In what follows, we will need a few basic notions from the ergodic theory of $C^*$-dynamical systems, which we next  recall.
A $C^*$-dynamical system is a triple
$(\ga, H, \beta)$, where $\ga$ is a (unital) $C^*$-algebra, $H$ is a (discrete) group, and
$\beta: H\rightarrow {\rm Aut}(\ga)$ is a group homomorphism
from $H$ to ${\rm Aut}(\ga)$, the group of all $*$-automorphisms of
$\ga$. A state $\om$ on $\ga$ is said to be invariant if  $\om\circ\beta_h=\om$ for all $h\in H$.
On the GNS representation $(\ch_\om, \pi_\om, \xi_\om)$  of any such state $\om$, the action $\beta$ of $H$ can be implemented unitarily. More precisely, for each $h\in H$, $U_h^\om \pi_\om(a)\xi_\om:=\pi_\om(\beta_h(a))\xi_\om$, $a\in\ga$, defines a unitary on $\ch_\om$ such that $U_h^\om\pi_\om(a)U_{h^{-1}}^{\om}=\pi_\om(\beta_h(a))$ for all
$a\in\ga$. Lastly, a $C^*$-dynamical system $(\ga, H, \beta)$ is said to be $H$-abelian if, for any
$H$-invariant state $\om$, the family
$E_\om \pi_\om(\ga) E_\om$ is abelian, where $E_\om$ is the orthogonal projection
onto  $\ch_\om^H:=\{\xi\in\ch_\om: U_h^\om \xi=\xi\,, \textrm{for all}\, h\in H\}$.
$H$-abeliannes is often proved to hold as a consequence of a stronger property, asymptotical abelianness, see {\it e.g.} \cite[Proposition 3.1.16]{S}.
A dynamical system $(\ga, H, \beta)$ is called asymptotically abelian if there exists
a sequence $\{h_n: n\in\bn\}\subset H$ such that  for all $a, b\in\ga$ one has
$$\lim_{n\rightarrow\infty} \|  [\beta_{h_n}(a), b] \|=0\, ,$$
where $[\cdot, \cdot]$ denotes the commutator.

\section{Spreadable states}

A state $\varphi$ on $\mathbb{A}_\a^\bz$ is said to be spreadable if 
$\varphi\circ\a_h=\varphi$, for any strictly increasing map $h:\bz\rightarrow\bz$.\\
We now consider the semigroup $\mathbb{J}_\bz$ of all increasing maps $h$ from $\bz$ to itself whose range is cofinite, that is $|\bz\setminus h(\bz)|<\infty$.
Since any increasing map coincides locally with an element of  $\mathbb{J}_\bz$ (namely the restriction to any finite subset of $\bz$ of $h$ is the same as a suitable element of  $\mathbb{J}_\bz$), a state
is spreadable if and only if it is invariant under the action of $\mathbb{J}_\bz$.\\
Unlike the semigroup of all increasing maps, $\mathbb{J}_\bz$ is countably generated. Precisely,
for every fixed $l\in\bz$, define the $l$-{right hand-side partial shift} as the element $\theta_l$ of $\bj_\bz$ given by
$$
\theta_l(k):=\left\{\begin{array}{ll}
                      k & \text{if}\,\, k<l\,, \\
                      k+1 & \text{if}\,\, k\geq l\,.
                    \end{array}
                    \right.
$$
As shown in \cite{ CFG2},  the set $\{\theta_l, \tau, \tau^{-1}: l\in\bz\}$  generates
$\bj_\bz$ as a unital semigroup, where $\tau$ is given by $\tau(k)=k+1$, $k\in\bz$.\\

Our next goal is to lift the action of $\bj_\bz$ on $\mathbb{A}_\a^\bz$ to an automorphic action
of a group $G$ containing $\bj_\bz$ on the larger $C^*$-algebra given by the non-commutative torus indexed
by $\bz\big[\frac{1}{2}\big]$, the set of dyadic numbers.
The idea is to extend the generators of $\bj_\bz$ to increasing bijections of $\bz\big[\frac{1}{2}\big]$.
We start by extending $\theta_0$ to a bijection $\widetilde{\theta}_0$ of $\bz\big[\frac{1}{2}\big]$, defined by
$$
\widetilde{\theta}_0(d):=\left\{\begin{array}{ll}
                      d & \text{if}\,\, d\leq -1\,, \\
                      2d+1&\text{if}\,\, -1\leq d\leq 0\\
                      d+1 & \text{if}\,\, d\geq 0\,.
                    \end{array}
                    \right.
$$
 For each natural $n$, we then consider the dilation $\delta_n$ by $2^n$ acting on $\bz\big[\frac{1}{2}\big]$, that is
$$\delta_n(d)=2^n d\,, \,\, d\in \bz\bigg[\frac{1}{2}\bigg]\,. $$
We then define  $\widetilde{\theta_n}:=\delta_n^{-1}\circ\widetilde{\theta}_0\circ\delta_n$ and
$\widetilde{\tau}_{k, n}(r)=r+\frac{k}{2^n}\,, \, r\in \bz\big[\frac{1}{2}\big]$, for $n\in\bn$ and $k\in\bz$.\\
Let  $G_n$ be the group generated by
$\widetilde{\theta_k}$ and $\widetilde{\tau}_{1, n}$ with $k=1, \ldots, n$.
$G_n$ is well defined as a subgroup of the group of all bijections of $\bz\big[\frac{1}{2}\big]$. Moreover, by definition $G_n\subset G_{n+1}$ for all $n$. Therefore, the union
$$
G:= \bigcup_{n} G_n\, .
$$
is a group acting  through automorphisms on  $\mathbb{A}_\a^{\bz\left[\frac{1}{2}\right]}$,
the noncommutative torus indexed by ${\bz\left[\frac{1}{2}\right]}$, which is the universal
$C^*$-algebra generated by unitaries $\left\{u_r: r\in{\bz\left[\frac{1}{2}\right]}\right \}$ satisfying
the commutation rules $u_ru_s=e^{i2\pi\a}u_su_r$,  for all $r, s$ in ${\bz\left[\frac{1}{2}\right]}$
with $r<s$. We denote by $\widetilde{{\rm tr}}$ its canonical trace.
Indeed, for every $g\in G$, by universality there exists a $*$-automorphism $\alpha_g$ of $ \mathbb{A}_\a^{\bz\left[\frac{1}{2}\right]} $  uniquely determined by
$$\alpha_g(u_d):=u_{g(d)} \,\,\, d\in\bz\left[\frac{1}{2}\right]\,. $$
Also note that by universality there exists a $*$-automorphism
$\Phi$ on $\mathbb{A}_\a^{\bz\left[\frac{1}{2}\right]}$ acting on the generators by mutiplying by $2$ the corresponding index, that is $\Phi(u_r)= u_{2r}$, for all $r$ in $\bz\left[\frac{1}{2}\right]$. Its inverse $\Phi^{-1}$ is
obviously given by $\Phi^{-1}(u_r)=u_{\frac{r}{2}}$, for all $r$ in $\bz\left[\frac{1}{2}\right]$.\\
For every $n$ in $\bn$, we denote by $\mathbb{A}_\a^{\frac{\bz}{2^n}}\subset \mathbb{A}_\a^{\bz\left[\frac{1}{2}\right]}$ the
$C^*$-subalgebra generated by $\{u_r: r\in \frac{\bz}{2^n}\}$, where $\frac{\bz}{2^n}$ is the set  $\{\frac{k}{2^n}: k\in\bz\}$.\\
We have $\Phi\left( \mathbb{A}_\a^{\frac{\bz}{2^ {n+1} } }\right)=\mathbb{A}_\a^{\frac{\bz}{2^ {n} } }$
and $\Phi^{-1}\left(\mathbb{A}_\a^{\frac{\bz}{2^ {n} } }\right)=\mathbb{A}_\a^{\frac{\bz}{2^ {n+1} } }$.
Furthermore, when $n=0$ the $C^*$ subalgebra $C^*(\{u_l: l\in\bz\})$ is still universal, as proved below.
\begin{prop}\label{copy}
 $C^*(\{u_l: l\in\bz\})\subset\mathbb{A}_\a^{\bz\left[\frac{1}{2}\right]}$ is canonically isomorphic with $\mathbb{A}_\a^\bz$.
\end{prop}
\begin{proof}
Denote by $w_l$ the generators of $\mathbb{A}_\a^\bz$. By universality of  $\mathbb{A}_\a^\bz$, 
there exists a $*$-epimorphism $p: \mathbb{A}_\a^\bz\rightarrow C^*(\{u_l: l\in\bz\})$ with
$p(w_l)=u_l$, for all $l\in\bz$, hence the equality $\widetilde{{\rm tr}}\circ p={\rm tr}$ holds,
where $\widetilde{{\rm tr}}$ is (the restriction to $C^*(\{u_l: l\in\bz\})$ of) the canonical trace of 
$\mathbb{A}_\a^{\bz\left[\frac{1}{2}\right]}$.\\
Injectivity of $p$ now follows at once. Indeed, if $x$ in $\mathbb{A}_\a^\bz$ is such that $p(x^*x)=0$, then
${\rm tr}(x^*x)=0$, hence $x=0$ by faithfulness of ${\rm tr}$.
\end{proof}
At this point, we are in a position to replicate  the proof of \cite[Proposition 3.1]{CDRCAR} {\it verbatim}, obtaining the following.
\begin{prop}
\label{prop:Ginv_spread}
The map sending any $G$-invariant state on $\mathbb{A}_\a^{\bz\left[\frac{1}{2}\right]}$ to its restriction 
 on $\mathbb{A}_\a^{\bz}$ establishes an affine homeomorphism between the compact convex sets of $G$-invariant states on $\mathbb{A}_\a^{\bz\left[\frac{1}{2}\right]}$ and spreadable states on $\mathbb{A}_\a^{\bz}$. 
\end{prop}

Our next goal is to show that the extended dynamical system is $G$-abelian. 
Before doing this, let us fix some notation.  Define the bicharacter $u_\a: \bz\times\bz\rightarrow\bt$ as
$$u_\alpha(k, l):= e^{i2\pi \a kl},\,\,\, k, l\in\bz\, .$$
The isotropy subgroup associated with $u_\alpha$ is the set
$$\Delta_\a:=\{k: u_\alpha(k, k)=1\}\, .$$
Note that if  $\frac{\a}{2\pi}$ is irrational, then $\Delta_\a=\{0\}$.
When $\frac{\a}{2\pi}$ is rational, $\Delta_\a$ is a non-trivial subgroup of $\bz$, as we show below.\\
For every integer $n$,  $\{\frac{n}{2}\}$ stands for $\frac{n}{2}$ is $n$ is even and $\frac{n+1}{2}$ if $n$ is odd.

\begin{lem}\label{subgroup}
If $\frac{\a}{2\pi}=\frac{q_1^{n_1}\ldots q_s^{n_s}}{p_1^{m_1}\ldots p_r^{m_r}}$, 
with $q_1, \ldots, q_s, p_1, \ldots, p_r$ being prime numbers different from each other and $n_1, \ldots, n_s, m_1, \ldots, m_r$ are in $\bn$, then $\Delta_\a=n_0\bz$, where $n_0=p_1^{\{\frac{m_1}{2}\}}\cdots p_r^{\{\frac{m_r}{2}\}}$.\\
Moreover, the restriction of $u_\alpha$ to $\Delta_\a\times \Delta_\a$ is identically $1$.
\end{lem}

\begin{proof}
We start by showing that $\Delta_\a$ is a subgroup of $\bz$.
If $k, l$ in $\Delta_\a$, then $k+l$ in  $\Delta_\a$.
Indeed, by definition, we have that $\frac{q_1^{n_1}\ldots q_s^{n_s}}{p_1^{m_1}\ldots p_r^{m_r}} k^2$ and $\frac{q_1^{n_1}\ldots q_s^{n_s}}{p_1^{m_1}\ldots p_r^{m_r}}l^2$ are integers. For this to happen, we must have $k=k' p_1^{\{\frac{m_1}{2}\}}\cdots p_r^{\{\frac{m_r}{2}\}}$
and $l=l'p_1^{\{\frac{m_1}{2}\}}\cdots p_r^{\{\frac{m_r}{2}\}}$, for some
$k', l'\in\bn$. Therefore, we have $kl=k'l' p_1^{m_1'}\cdots {p_r}^{m_r'} $ divides $n$ as $m_1'\geq m_1, \ldots, m_r'\geq m_r$, and we are done.
\end{proof}

Let $\varphi$ be a $G$-invariant state on $\mathbb{A}_\a^{\bz\left[\frac{1}{2}\right]}$ . Associated with any $g$ in $G$ there is a unitary $T_g$ acting on the Hilbert space
$\ch_\varphi$ of the GNS representation of $\varphi$ as
$$
T_g^\varphi\pi_\varphi(a)\xi_\varphi:=\pi_\varphi(\a_g(a))\xi_\varphi\,, \,\, a\in\mathbb{A}_\a^{\bz\left[\frac{1}{2}\right]} \,.
$$
We denote by $E_\varphi$ the orthogonal projection onto the closed subspace 
$\ch_\varphi^G:=\{\xi\in\ch_\varphi: T_g^\varphi\xi=\xi, \,\textrm{for all}\,\, g\in G \}$. Notice that
 $\bc\xi_\varphi\subset\ch_\varphi^G$, where
$\xi_\varphi$ is the GNS vector of $\varphi$.\\

\begin{lem}
\label{lem:Gabelian}
The $C^*$-dynamical system $\left(\mathbb{A}_\a^{\bz\left[\frac{1}{2}\right]}, G, \alpha\right)$ is
$G$-abelian.
\end{lem}

\begin{proof}
We have to show that for any $G$-invariant state $\varphi$ the set $E_\varphi \pi_\varphi(\mathbb{A}_\a^{\bz\left[\frac{1}{2}\right]}) E_\varphi$ is commutative.
The proof will be done in two steps.
We first deal with elements $a$ belonging to some spectral eigenspace with $\partial a\in\bz\setminus\Delta_\a$.\\
We aim to show that  we have 
\begin{equation}\label{oddcomponent}
E_\varphi\pi_\varphi(a)E_\varphi=0,\,\,\, \textrm{for all}\,\, a\,\,{\rm with}\,\, \partial a\in\bz\setminus\Delta_\a\,.
\end{equation} 
To this end, let us
introduce $F_\varphi$, the orthogonal projection onto the closed subspace $\{\xi\in\ch_\varphi: T^\varphi_{g_n}\xi=\xi\,, \textrm{for all}\, n\in\bn\}$. Note that $E_\varphi\leq F_\varphi$.\\
For $k, l\in\bz$, define 
$$ [x, y]_{k, l}:=xy-u_\a (k, l)yx, \, x, y\in \mathbb{A}_\a^{\bz\left[\frac{1}{2}\right]}$$
For any homogeneous $a$ sitting in $(\mathbb{A}_\a^{\bz\left[\frac{1}{2}\right]})^{\partial a}$, we have
\begin{align*}
&[F_\varphi\pi_\varphi(a )F_\varphi, F_\varphi\pi_\varphi(a^*)F_\varphi]_{(\partial a, -\partial a)}\\
&=\lim_{n\rightarrow\infty} \frac{1}{n+1}\sum_{j=0}^n
F_\varphi \pi_\varphi([\a_{\widetilde{\tau}_{j, 0}}(a), a^*]_{\partial a, -\partial a})F_\varphi=0\,, 
\end{align*}
where $\widetilde{\tau}_{j, 0}$ is the monotone map from $\bz\left[\frac{1}{2}\right]$ to itself
given by $\widetilde{\tau}_{j, 0}(r)=r+j$, for all $r$ in $\bz\left[\frac{1}{2}\right]$.
Indeed, the last equality above can be obtained exactly as in the proof of \cite[Lemma 3.3]{CDRCAR}.
In other words, we have proved 
$$F_\varphi\pi_\varphi(a )F_\varphi\pi_\varphi(a^*)F_\varphi=u_\a(k, -k)F_\varphi\pi_\varphi(a^* )F_\varphi\pi_\varphi(a)F_\varphi\, .$$
Now, since $u_\a(k, -k)$ is a non-trivial phase, the spectrum of $F_\varphi\pi_\varphi(a )F_\varphi\pi_\varphi(a^*)F_\varphi$ is $\{0\}$, hence $F_\varphi\pi_\varphi(a )F_\varphi\pi_\varphi(a^*)F_\varphi=0$, and thus $F_\varphi\pi_\varphi(a )F_\varphi=0$. Therefore, from $E_\varphi\leq F_\varphi$ it follows that 
$E_\varphi\pi_\varphi(a) E_\varphi=E_\varphi F_\varphi\pi_\varphi(a) F_\varphi E_\varphi=0$.\\
The second step to take is to show that $[E_\varphi\pi_\varphi(a)E_\varphi,  E_\varphi\pi_\varphi(b)E_\varphi]=0$
for all $a, b$ with $\partial a, \partial b\in \Delta_\a$. 
Since $\Delta_\a$ is a subgroup of $\bz$ by Lemma \ref{subgroup}, 
we have that $\mathcal{A}\coloneq\overline{{\rm span}}\{a\in \mathbb{A}_\a^{\bz\left[\frac{1}{2}\right]}: \partial a\in\Delta_\a \}$
is a $C^*$-subalgebra of our non-commutative torus.
We can then consider the restriction of the dynamics to this subalgebra. We claim that the dynamical
system thus obtained is asymptotically abelian, in that  for any $a, b\in\mathcal{A}$ one has
$$\lim_{n\rightarrow\infty}\|[\a_{\widetilde{\tau}_{n, 0}}(a), b]\|=0\, .$$
where $\widetilde{\tau}_{n, 0}$ is the monotone map defined above.\\
Now, by applying \cite [Proposition 3.1.16]{S} we see that the restricted system is $G$-abelian,
that is $P_\varphi\pi_\varphi(a)P_\varphi \pi_\varphi(b)P_\varphi=P_\varphi\pi_\varphi(b)P_\varphi \pi_\varphi(a)P_\varphi $, for all
$a, b\in \mathcal{A}$, where $P_\varphi\leq E_\varphi$ is the orthogonal projection onto 
$\{\eta\in \overline{\pi_\f(\mathcal{A})\xi_\varphi} : T_g\eta=\eta, g\in G\}$.\\
We next show that  the equality $P_\f=E_\f$ actually holds.
To this end, define 
$$\ch_{\f, +}\coloneq\overline{\pi_\f(\mathcal{A})\xi_\f},\quad\ch_{\f, -}:=\overline{\rm{span}}\{\pi_\f(a)\xi_\f\,,\,\, a: \partial a\in \bz\setminus\Delta_\a\}\,.$$
Now $\ch_\f$ decomposes as $\ch_\f=\ch_{\f, +}\oplus\ch_{\f, -}$.\\
To see this, first note that if $k$ lies in $\Delta_\a$ and $l$ in $\bz\setminus\Delta_\a$, then
$k+l$ continues to lie in $\bz\setminus\Delta_\a$ as readily follows from Lemma \ref{subgroup}.
But then the orthogonality between $\ch_{\f, +}$ and $\ch_{\f, -}$
follows at once  from \eqref{oddcomponent}.\\
Thanks to the first step, it is now easy to verify that the restriction of $E_\f$ to $\ch_{\f, -}$ is zero, {\it i.e.} $E_\f=P_\f$, as maintained.\\
Putting the two steps together, we finally reach the conclusion, for $\mathbb{A}_\a^{\bz\left[\frac{1}{2}\right]}$
decomposes into the direct sum of $\mathcal{A}$ and $\oplus_{m\in \bz\setminus\Delta_\a} (\mathbb{A}_\a^{\bz[\frac{1}{2}]})^m$.\\
To conclude, the claim follows easily by a standard density argument after noting that it is certainly true if
$a, b$ are homogeneous with $\partial a, \partial b$ in $\Delta_\a$.
\end{proof}

We define $\Delta_\a^\perp=\{z\in\bt: z^l=1\,\, \textrm{for all}\,\, l\in\Delta_\a\}$.
Note that ${\Delta_\a^\perp}=\bt$ if $\frac{\a}{2\pi}$ is irrational, whereas ${\Delta_\a^\perp}=\{z\in\bt: z^{n_0}=1\}$
for $\frac{\a}{2\pi}=\frac{q_1^{n_1}\ldots q_s^{n_s}}{p_1^{m_1}\ldots p_r^{m_r}}$ with
$n_0=p_1^{\{\frac{m_1}{2}\}}\cdots p_r^{\{\frac{m_r}{2}\}}$.
Since ${\Delta_\a^\perp}\subset\bt$, the gauge action induces an automorphic action of ${\Delta_\a^\perp}$ on
 $\mathbb{A}_\a^\bz$.

\begin{cor}
\label{cor:sprGinv}
Any spreadable state  on $\mathbb{A}_\a^\bz$ is automatically ${\Delta_\a^\perp}$-invariant.
In particular, if $\frac{\a}{2\pi}$ is irrational, all spreadable states are gauge invariant. 
\end{cor}

\begin{proof}
Thanks to Proposition \ref{prop:Ginv_spread}, we may as well prove the statement for
$G$-invariant states on $\mathbb{A}_\a^{\bz\left[\frac{1}{2}\right]}$.
Let $\varphi$ be any such state; we need to show that $\varphi\circ\gamma_z=\varphi$ for
all $z$ in ${\Delta_\a^\perp}$. Since the set of all homogeneous elements spans a dense subalgebra, it suffices to
verify the equality only on homogeneous elements
For any such $a$, we have
$\varphi\circ\gamma_z\, (a)=\varphi(\gamma_z(a))= \varphi (z^{\partial a} a)=z^{\partial a}\varphi (a)$.
Now, if $\partial a$ is in $\Delta_\a$, then $z^{\partial a}=1$ by definition of ${\Delta_\a^\perp}$, and the equality is
certianly satisfied. If $\partial a$ is in $\bz\setminus\Delta_\a$, then from 
\eqref{oddcomponent}, that is $E_\varphi\pi_\varphi(a)E_\varphi=0$, it immediately follows that both
$\varphi(a)$ and $\varphi(\gamma_z(a))$ vanish. The proof is thus complete.
\end{proof}
Let us  now consider $S\subset G$  the unital semigroup generated by
$\widetilde{\theta_n}$ and $\widetilde{\tau}_{k, n}$, with $n\in\bz$ and $k\in\bn$. 
Note that $h(r)\geq r$, for all $r$ in $\bz\big[\frac{1}{2}\big]$, if $h$ sits in $S$.\\
For  every $G$-invariant state $\varphi$ on $ \mathbb{A}_\a^{\bz\left[\frac{1}{2}\right]}$, define
$\ch_\varphi^S:=\{\xi\in\ch_\varphi: T_h^\varphi\xi=\xi,\, h\in S\}$. Clearly
$\ch_\varphi^G$ is a subspace of $\ch_\varphi^S$. However, since the group generated by $S$ in $G$ is $G$ itself, it is easy to see that the equality
$\ch_\varphi^S=\ch_\varphi^G$ holds, for any $G$-invariant state, {\it cf.} \cite[Lemma 3.2]{CDRCAR}.

\begin{lem}\label{factorisation}
Any extreme spreadable state $\varphi$ on $\mathbb{A}_\a^\bz$ factorizes as
$$\varphi (u_{j_1}^{k_1}\cdots u_{j_l}^{k_l})=\om(u_{j_1}^{k_1})\cdots\om(u_{j_l}^{k_l})$$
for all $j_1\neq\cdots \neq j_l, k_1, \cdots, k_l$ in $\bz$, for a suitable state
$\om$ on $C(\bt)$.
\end{lem}
\begin{proof}
Let $\varphi$ be an extreme $G$-invariant state on $\mathbb{A}_\a^{\bz\left[\frac{1}{2}\right]}$. By Lemma \ref{lem:Gabelian}, and \cite[Proposition 3.1.12]{S} $E_{\varphi}^{S}$ is the rank-one projection onto $\bc\xi_\varphi$.\\
Consider now the family of unitaries
$\{T_h^\varphi: h\in S\}\subseteq\cb(\ch_\varphi)$. By \cite[Proposition 4.3.4.]{BR1}, there exists a net  $\{S_\gamma: \gamma\in J\}$, whose terms are finite convex combinations of the form $S_\gamma=\sum_{k=1}^{n_\gamma} \lambda_{k}^\gamma  T_{h_k^\gamma}^\varphi$, for some $\lambda_k^\gamma\geq 0$ with $\sum_{k=1}^{n_\gamma} \lambda_k^\gamma=1$, which is strongly convergent to  $E_\varphi^S$. \\
Fix $a$,$b\in  \mathbb{A}_\a^{\bz\left[\frac{1}{2}\right]}$. For each $\g\in J$, define $b_\gamma:=\sum_{k=1}^{n_\gamma} \lambda_{k}^\gamma  \a_{h_k^\gamma}(b)$. We have
\begin{align}
\begin{split}
\label{eq:b_gamma}
\lim_\g \varphi(ab_\g)&= \lim_\g \langle \pi_\varphi(a b_\g) \xi_\varphi, \xi_\varphi \rangle\\
&= \lim_\g  \langle  \pi_\varphi(a)  \sum_{k=1}^{n_\g}\lambda_{k}^\g T_{h_k^\g}^\varphi \pi_\varphi(b)\xi_\varphi, \xi_\varphi\rangle
\\
&=\langle  \pi_\varphi(a)   E_{\bc\xi_\varphi}\pi_\varphi(b)\xi_\varphi, \xi_\varphi\rangle=
\varphi(a)\varphi(b)\,.
\end{split} .
\end{align} 
 For any $l$, we denote by $\iota_l: C(\bt)\rightarrow C^*(u_l)$ the continuous functional calculus of $u_l$, that is $\iota_l(f):=f(u_l)$, for all
$f$ in $C(\bt)$.
 Let now $\om$ be the state on $C(\bt)$ defined by $\om(f):=\varphi(\iota_l(f))$, $f\in C(\bt)$. The spreading-invariance of $\varphi$ ensures that $\om$ is well defined since its definition does not depend on $l$. 
Our goal is to show that $\varphi$ factorizes as
$$\varphi(\iota_{j_1}(f_1)\iota_{j_2}(f_2)\cdots \iota_{j_n}(f_n))=\om(f_1(u_{j_1}))\om(f_2(u_{j_2}))\cdots\om(f_n(u_{j_n}))$$
for every $n\in\bn$, for every $j_1\neq \ldots \neq  j_n\in\bz$, and for every $f_1, f_2, \ldots, f_n$ in $C(\bt)$. We proceed by induction on $n$. For $n=1$, the equality is true. Suppose now that the assertion holds with $n$ and we will prove it holds with $n+1$. \\
Let $n\in\bn$, $j_1<j_2<\ldots<j_{n+1}$ in $\bz$, and $f_1, \cdots, f_n, g\in C(\bt)$. Since $\varphi$ is $G$-invariant on $\mathbb{A}_\a^{\bz\left[\frac{1}{2}\right]}$, we have
$$
\varphi(\iota_{j_1}(f_1)\cdots \iota_{j_n}(f_n)\iota_{j_{n+1}}(g))=\varphi(\iota_{j_1}(f_1)\ldots \iota_{j_n}(f_n) \iota_{h_k^\g(j_{n+1})}(g))\,,
$$
where, for each $\g$ in $J$, the $h_k^\g$'s are the monotone functions in $S$ introduced in the definition of the net $S_\g$.  
Therefore, after summing on all $k$'s between $1$ and $n_\g$, we find
\begin{align*}
&\varphi(\iota_{j_1}(f_1)\cdots \iota_{ j_n}(f_n)\iota_{j_{n+1}}(g))\\
=&\sum_{k=1}^{n_\g}\lambda_k^\g\varphi(\iota_{j_1}(f_1)\cdots \iota_{j_n}(f_n) \iota_{h_k^\g(j_{n+1})}(g))\\=&\varphi(\iota_{j_1}(f_1)\cdots \iota_{j_n}(f_n)b_\g)
\end{align*}
\medskip
where $b_\g:=\sum_{k=1}^{n_\g}\lambda_k^\g \iota_{h_k^\g(j_{n+1})}(g)=\sum_{k=1}^{n_\g} \lambda_{k}^\g \a_{h_k^\g}(\iota_{j_{n+1}}(g))$.\\
Applying \eqref{eq:b_gamma} to $b=\iota_{j_n+1}(g)$ we obtain
$$
\lim_\g\varphi(\iota_{j_1}(f_1)\cdots \iota_{j_n}(f_n)b_\g)=\varphi(\iota_{j_1}(f_1)\cdots \iota_{j_n}(f_n))\om(g)\,.
$$
Finally, by the inductive hypothesis, it follows that 
$$
\varphi(\iota_{j_1}(f_1)\cdots \iota_{j_n}(f_n))\om(g)=\om(f_1)\cdots\om(f_n)\om(g)\,,
$$
and we are done.
\end{proof}
The next result provides a supply of product states on the rational non-commutative torus.
\begin{lem}\label{existencepro}
Let $\frac{\a}{2\pi}=\frac{q_1^{n_1}\ldots q_s^{n_s}}{p_1^{m_1}\ldots p_r^{m_r}}$ be a rational number and 
$n_0:=p_1^{\{\frac{m_1}{2}\}}\cdots p_r^{\{\frac{m_r}{2}\}}$.\\ 
For any state $\om$ on $C(\bt)$ with
$\om(z^l)=0$ if $l$ is not a multiple of $n_0$, there exists a state $\times\om$ on $\mathbb{A}_\a^\bz$ uniquely determined by
$$\times\om\, (u_{i_1}^{l_1} \cdots u_{i_l}^{k_l})=\om(u_{i_1}^{l_1})\cdots\om (u_{i_l}^{k_l})$$
for all $i_1\neq\ldots\neq i_l$ in $\bz$ and all $k_1, \ldots, k_l$ in $\bz$.
\end{lem}

\begin{proof}
Let ${\Delta_\a^\perp}\subset\bt$ the subgroup of the $n_0$-th roots of unity, {\it i.e.} $\{ z\in\bt: z^{n_0}=1\}$.
The infinite product $(\Delta_\a^\perp)^{\bz}\subset\bt^\bz$ is a compact group, and we denote by
$\mu$ its Haar measure. Averaging the action of $({\Delta_\a^\perp})^\bz$ on $\mathbb{A}_\a^\bz$,
we get a (faithful) conditional expectation  
$$F(a):=\int_{{(\Delta_\a^\perp)}^{\bz}} \gamma_\mathfrak{z}(a) {\rm d}\mu(\mathfrak{z})\,, a\in\mathcal{A}_\a^\bz$$
onto the fixed-point subalgebra, which is easily seen to be given by 
$$\overline{{\rm span}}\{u_{i_1}^{n_0k_1}\cdots u_{i_j}^{n_0k_j}: i_1, \ldots, i_j, k_1, \ldots, k_j\in\bz\}\, .$$
But thanks to the commutation rules of the infinite torus, one immediately sees that
all monomials above commute with one another. Therefore, the fixed-point subalgebra
identifies with the (maximal) infinite tensor product $\otimes_{i\in\bz} C(u_i^{n_0})$ by a minor variation of  the proof
of Proposition \ref{copy}.
Because the Haar measure $\mu$ is the product of the normalized counting measure of ${\Delta_\a^\perp}$, one sees at once that $F$ is multiplicative on the product of two elements lying in two different factors of the tensor product, {\rm i.e.} $F(u_{i}^{k_i}u_j^{k_j})=F(u_{i}^{k_i})F(u_j^{k_j})$, if $i\neq j$  and for all $k_i, k_j$ in $\bz$.
Now any state $\om$ on $C(\bt)$ with the property in the statement writes
as $\om=\om'\circ F_0$, where $F_0$ is the conditional expectation 
obtained by averaging the action of ${\Delta_\a^\perp}$ on $C(\bt)$, for a unique state $\om'$ on the range of $F_0$.\\
We can then consider the infinite product $\otimes^\bz \om'$, as a state of the fixed-point subalgebra.
The composition $\otimes^\bz\om'\,\circ F$ will then yield the sought product $\times\om$.
\end{proof}
We are ready to state the main result of the paper.
\begin{thm}\label{spreadability}
The set of all spreadable states on $\mathbb{A}_\a^\bz$ is the Bauer simplex whose extreme points are
the product states $\times\om$, where $\om$ is a $\Delta_\a^\perp$-invariant state on $C(\bt)$.\\
In particular, we have:
\begin{itemize}
\item[(i)] If $\frac{\a}{2\pi}$ is irrational, the trace is the only spreadable state;
\item[(ii)] If  $\frac{\a}{2\pi}=\frac{q_1^{n_1}\ldots q_s^{n_s}}{p_1^{m_1}\ldots p_r^{m_r}}$ is rational, 
$\om$ is any state on $C(\bt)$ invariant under all rotations by $n_0$-th roots of unity, where
$n_0=p_1^{\{\frac{m_1}{2}\}}\cdots p_r^{\{\frac{m_r}{2}\}}$.
\end{itemize}
\end{thm}

\begin{proof}
We show that any extreme spreadable state $\varphi$ on $\mathbb{A}_\a^\bz$ is of the type above.
By Lemma \ref{factorisation} any such $\varphi$ factorizes as
$\varphi(u_{j_1}^{k_1}\cdots u_{j_l}^{k_l})=\om(u_{j_1}^{k_1})\cdots\om(u_{j_l}^{k_l})$ for all $j_1\neq\cdots \neq j_l, k_1, \cdots, k_l$ in $\bz$,
where $\om$ is a $\Delta_\a^\perp$-invariant state on $C(\bt)$ thanks to Corollary \ref{cor:sprGinv}.\\
If $\frac{\a}{2\pi}$ is irrational, the only such $\om$ is the normalized Lebesgue measure, hence
$\varphi$ is the canonical trace.\\
If $\frac{\a}{2\pi}$ is rational, first note $\Delta_\a^\perp$ is $\{z\in\bt: z^{n_0}=1\}$. The state $\varphi$
must be a product $\times \om$, with $\om$ a  $\Delta_\a^\perp$-invariant state on $C(\bt)$, and such product states exist by Lemma \ref{existencepro}.\\
To conclude, we only need to make sure that our extreme set is compact. But this follows  from the continuity of the bijection
$\om\mapsto\times \om$, $\om$ a $\Delta_\a^\perp$-invariant state.
\end{proof}
\section{Shift-invariant states}
We denote by $\tau$ the $*$-automorphism of  $\mathbb{A}_\a^\bz$ given by
$\tau(u_l)\coloneq u_{l+1}$, for all $l$ in $\bz$. In the sequel, $\tau$ will
be referred to as the shift.\\
The aim of this section is to prove that the compact convex set of all shift-invariant states on
$\mathbb{A}_\a^\bz$ is the Poulsen simplex irrespective of whether $\frac{\a}{2\pi}$ is rational or irrational.
\noindent
In order to accomplish this task, we first need to set some notation. For every $n\geq 1$, we denote  by $\mathcal{C}_{\tau^n}$ the compact convex set of
all $\tau^n$-invariant states on $\mathbb{A}_\a^\bz$.
The maps $\Phi_n: \mathcal{C}_{\tau^n}\rightarrow\mathcal{C}_{\tau^n}$
given by $\Phi_n(\om)=\om\circ\tau$, $\om\in\mathcal{C}_{\tau^n}$, are clearly
affine bijections satisfying  $\Phi_n^n={\rm id}_{\mathcal{C}_{\tau^n}}$.
In particular, the restriction of $\Phi_n$ to ${\rm Extr}(\mathcal{C}_{\tau^n})$, the set of
all extreme points of $\mathcal{C}_{\tau^n}$, is a bijection of  ${\rm Extr}(\mathcal{C}_{\tau^n})$.\\

We start with a couple of prepatory results.
\begin{prop} \label{tausimplex}
For every $n\geq 1$, the set $\mathcal{C}_{\tau^n}$ is a metrizable Choquet simplex. Moreover, any $\om$ in  $\mathcal{C}_{\tau^n}$ is
 $\Delta_\a^\perp$-invariant.
\end{prop}
\begin{proof}
 $\mathcal{C}_{\tau^n}$ is certainly metrizable because even the whole set of all states of
the infinite torus is so.  The fact that $\mathcal{C}_{\tau^n}$ is a Choquet simplex follows
 from $\bz$-abelianity of the corresponding system, which has essentially been shown over the course of the proof of 
Lemma \ref{lem:Gabelian}.  Finally, the invariance under $\Delta_\a^\perp$ of any state in
$\mathcal{C}_\tau$ can be derived as in the proof of Corollary \ref{cor:sprGinv}.
\end{proof}
\begin{prop}\label{decompo}
Let $n\geq 1$ be a fixed natural number. 
A state $\om$ in $\mathcal{C}_{\tau^n}$ lies in $\mathcal{C}_\tau$
if and only if its barycentric measure is $\Phi_n$-invariant. In particular, the map
$\mathcal{C}_{\tau^n}\ni\om \mapsto \frac{1}{n}\sum_{k=0}^{n-1}\om\circ\tau^k\in \mathcal{C}_\tau$ sends
the set of extreme states of $\mathcal{C}_{\tau^n}$ onto the set of the extreme states of 
 $\mathcal{C}_\tau$.
\end{prop}
\begin{proof}
Let $\om$ be any state in  $\mathcal{C}_{\tau^n}$ and let $\om=\int_{{\rm Extr}(\mathcal{C}_{\tau^n})}\varphi {\rm d}\mu(\varphi)$ be
its barycentric decomposition, where $\mu$ is a probability measure supported  on ${\rm Extr}(\mathcal{C}_{\tau^n})$, the set of
extreme points of $\mathcal{C}_{\tau^n}$, which is a $G_\delta$-subset of $\mathcal{C}_{\tau^n}$ by
the general lemma of Choquet, see \cite[3.4.1]{S}.\\
We claim that  $\om\circ\tau$ decomposes as $\om\circ\tau=\int_{{\rm Extr}(\mathcal{C}_{\tau^n})}\varphi {\rm d}\widetilde{\mu}(\varphi)$,
where $\widetilde{\mu}$ is given by $\widetilde{\mu}(E)\coloneq\mu(\Phi_n^{-1}(E))$, for any Borel subset  $E\subset {\rm Extr}(\mathcal{C}_{\tau^n})$.
Therefore, $\om$ is shift invariant if and only if  
$$\int_{{\rm Extr}(\mathcal{C}_{\tau^n})}\varphi {\rm d}\mu(\varphi)=\int_{{\rm Extr}(\mathcal{C}_{\tau^n})}\varphi {\rm d}\widetilde{\mu}(\varphi)\,,$$
which is possible if and only if $\mu=\widetilde{\mu}$ by the uniqueness of the decomposition.\\
All we need to do is prove the claim. By definition of barycentric decomposition, the equality 
$f(\om)=\int_{{\rm Extr}(\mathcal{C}_{\tau^n})}f(\varphi) {\rm d}\mu(\varphi)$ holds for any
continuous affine function $f: \mathcal{C}_{\tau^n}\rightarrow\br$. Now, for any such $f$, the composition $f\circ\Phi_n$ is still
affine and continuous, which means 
\begin{align*}
f(\om\circ\tau)&=f\circ\Phi_n\, (\om)=\int_{{\rm Extr}(\mathcal{C}_{\tau^n})}f\circ\Phi_n\,(\varphi) {\rm d}\mu(\varphi)\\
&=\int_{{\rm Extr}(\mathcal{C}_{\tau^n})}f(\varphi) {\rm d}\widetilde{\mu}(\varphi)\,.
\end{align*}
Since the above equality holds for any continuous affine function $f$, we have 
$\om\circ\tau=\int_{{\rm Extr}(\mathcal{C}_{\tau^n})}\varphi {\rm d}\widetilde{\mu}(\varphi)$, as maintained.
For the second part, a shift-invariant state is extreme if and only if
its representing measure (as an element of the bigger convex set) is supported on
a minimal $\Phi_n$-invariant set. But minimal  $\Phi_n$-invariant sets are exactly
the orbits of $\Phi_n$, that is sets of the form $\{\om, \om\circ\tau, \ldots, \om\circ\tau^{n-1}\}$ for some
$\om$ in $\mathcal{C}_{\tau^n}$.\\
\end{proof}

%
We are ready to state the main result of this section.
\begin{thm}\label{Poulsen}
 $\mathcal{C}_\tau$ is (affinely homeomorphic with) the Poulsen simplex.
\end{thm}

\begin{proof}
 $\mathcal{C}_\tau$ metrizable by Proposition \ref{tausimplex}. Therefore, thanks to
\cite[Theorem 2.3]{Lind} it is enough to show  that  ${\rm Extr}(\mathcal{C}_{\tau})$ is dense
in $\mathcal{C}_{\tau}$, and that $\mathcal{C}_\tau$ is not trivial.\\
Both properties can be seen by a suitable adaptation of a technique  of
Parthasarathy, see \cite[Theorem 3.1]{Part}.\\
To begin with, for every fixed $n\geq 1$, let us define the countable family of $C^*$-subalgebras $\mathcal{B}_r\coloneq C^*(u_l:  -n+r(2n+1)\leq l\leq  r(2n+1)+n)$, $r$ in $\bz$.\\
We start by showing density of ${\rm Extr}(\mathcal{C}_{\tau})$
in $\mathcal{C}_{\tau}$. To this end, let $\varphi$ be any state in $\mathcal{C}_{\tau}$. Denote by $\varphi_0$ the restriction of $\varphi$ to $\mathcal{B}_0$.
Note that, for each $r$ in $\bz$, $\tau^{r(2n+1)}$ sends $\mathcal{B}_0$ onto $\mathcal{B}_r$, meaning
that each $\mathcal{B}_r$ is $^*$-isomorphic with $\mathcal{B}_0$. This allows us to think of
$\varphi_0$ as a state defined on each $\mathcal{B}_r$. To ease the notation, though, we will continue to denote by $\varphi_0$ the state $\varphi_0\circ\tau^{-r(2n+1)}$.\\
We claim that there exists a state
$\om_n$ on $\mathbb{A}_\a^\bz$ such that
$$\om_n(x_1  \dots x_\ell)=\varphi_0(x_1)\cdots\varphi_0(x_\ell)\,$$
for all $x_1\in\mathcal{B}_{r_1}, \ldots, x_\ell\in \mathcal{B}_{r_\ell}$.
By construction, $\om_n$ is invariant under the action of $\tau^{2n+1}$. Moreover, it is seen at once to be strongly clustering w.r.t. $\tau^{2n+1}$, {\it i.e.}
$\lim_k \om_n(a\tau^{k(2n+1)}(b))=\om_n(a)\om_n(b)$ for all $a, b$ in $\mathbb{A}_\a^\bz$.
In particular, $\om_n$ is extreme in the compact convex set of all $\tau^{2n+1}$-invariant states
by  \cite[Theorem 4.3.22]{BR1} and \cite[Proposition 3.1.10]{S}.\\
For every $n\geq 1$, define $\varphi_n\coloneq \frac{1}{2n+1}\sum_{k=-n}^n \om_n\circ \tau^k$.
By definition, the state $\varphi_n$ is shift invariant. More importantly, it is extreme in $\mathcal{C}_\tau$
thanks to Proposition \ref{decompo}. The next thing to do is to show that the sequence $\{\varphi_n: n\geq 1\}$
$*$-weakly converges to $\varphi$.\\
By a standard density argument, it is enough to verify that, for every $s>0$, one has the sought convergence  on
monomials of the type $x=u_{l_1}^{n_1}\cdots u_{l_j}^{n_j}$ with $-s\leq l_1<\ldots<l_j \leq s$.
\begin{align*}
|\varphi_n(x)-\varphi(x)|= \left|\frac{1}{2n+1}\sum_{k=-n}^n\om_n(\tau^k(x)) -\varphi(x) \right|\leq \frac{4s}{2n+1}\xrightarrow[n]{}0\, .
\end{align*}
We now move on to prove the claim. Denote by
$\mathcal{B}_r^{\Delta_\a^\perp}\subset\mathcal{B}_r$ the sublagebra invariant under the action of
$\Delta_\a^\perp$ (we recall that $\Delta_\a^\perp=\bt$ if $\frac{\a}{2\pi}$ is irrational, and 
$\Delta_\a^\perp=\{z\in\bt: z^{n_0}=1\}$ if $\frac{\a}{2\pi}$ is rational).
Note that $\mathcal{B}_r^{\Delta_\a^\perp}$ and  $\mathcal{B}_s^{\Delta_\a^\perp}$ commute for any $r\neq s$.
Therefore, the $C^*$-subalgebra generated by all the subalgebras $\mathcal{B}_r^{\Delta_\a^\perp}$ is isomorphic with the (maximal) tensor product $\otimes_{r\in\bz} \mathcal{B}_r^{\Delta_\a^\perp}$
Therefore, the fixed-point subalgebra
identifies with the (maximal) infinite tensor product $ \otimes_{r\in\bz} \mathcal{B}_r^{\Delta_\a^\perp}$ by a minor variation of  the proof
of Proposition \ref{copy}.
Averaging the action of the infinite product of $\Delta_\a^\perp$ with itself, one obtains
a conditional expectation $F: \mathbb{A}_\a^\bz\rightarrow \otimes_{r\in\bz} \mathcal{B}_r^{\Delta_\a^\perp}\cong C^*(\mathcal{B}_r^{\Delta_\a^\perp}: r\in\bz)$, which is
 multiplicative on the product of two elements lying in two different factors of the tensor product.
The state $\om$ can be obtained exactly as in Lemma \ref{existencepro} by composing the usual product state on a tensor product with $F$, because shift-invariant states are $\Delta_\a^\perp$-invariant  thanks to
Proposition \ref{tausimplex}.\\
We conclude by observing that $\mathcal{C}_\tau$ is never trivial (even when $\a$ is irrational), {\it i.e.} it is not the singleton
of the canonical trace. Indeed, it is enough to fix some $n>0$ and pick a
$\Delta_\a^\perp$-invariant state $\varphi$ on $\mathcal{B}_0$ different from (the restriction to $\mathcal{B}_0$ of) the canonical trace.
As above, there exists a state $\om$ on $\mathbb{A}_\a^\bz$ such that
$$\om(x_1  \dots x_\ell)=\varphi(x_1)\cdots\varphi(x_\ell)\,$$
for all $x_1\in\mathcal{B}_{r_1}, \ldots, x_\ell\in \mathcal{B}_{r_\ell}$
(we are thinking of $\varphi$ as a state defined on all $\mathcal{B}_r$'s as we did in the first part of the proof).
Clearly, $\om$ is $\tau^{2n+1}$-invariant and different from the canonical trace.  Now, the state $\frac{1}{2n+1}\sum_{k=-n}^n \om\circ\tau^k$ is 
shift-invariant and different from the trace, for the equality $\frac{1}{2n+1}\sum_{k=-n}^n \om\circ\tau^k={\rm tr}$ would imply
$\om\circ\tau^k=\om$ for all $-n\leq k\leq n$ as ${\rm tr}$ is strongly clustering w.r.t $\tau^{2n+1}$, which is not the case we are in.
\end{proof}
We conclude the section by proving that spreadable states make up a face of the larger
convex set of all shift-invariant states.
\begin{prop}
The convex set of all spreadable states on $\mathbb{A}_\a^\bz$ is a face of the set of all shift-invariant states. 
\end{prop}
\begin{proof}
Arguing as in the proof of \cite[Theorem 3.5]{CDRCAR}, it suffices to show that any extreme spreadable state is extreme in the bigger convex
of all shift-invariant states. By Theorem \ref{spreadability}, extreme spreadable states are product states, and such states are strongly clustering
w.r.t. the shift.
\end{proof}

\section*{Acknowledgments}
\noindent
All authors acknowledge  the support of the Italian INDAM-GNAMPA Project Code CUP\_E55F22333270001 and
the Italian PNRR MUR project PE0000023-NQSTI.\\



\begin{thebibliography}{9999}



\bibitem{BR1} Bratteli O.,  Robinson D. W. \emph{Operator algebras and
quantum statistical mechanics. I.}, Second edition, Texts and Monographs in Physics. Springer-Verlag, Berlin, 1997



\bibitem{CDMR} Crismale V., Del Vecchio S., Monni T.,  Rossi S. {\it Freedman's theorem for unitarily invariant states}, Commun. Math. Phys. \textbf{405} (2024), Paper No. 40.


\bibitem{CDRCAR} Crismale V., Del Vecchio S.,  Rossi S. {\it On distributional symmetries on the CAR algebra} \url{https://doi.org/10.48550/arXiv.2209.05431}


\bibitem{CFCMP} Crismale V., Fidaleo F.  {\it De Finetti theorem on the CAR algebra},
Commun. Math. Phys. \textbf{315} (2012), 135--152.

\bibitem{CriFid} Crismale V., Fidaleo F. {\it Symmetries and ergodic properties in quantum probabilty}, Colloq. Math.  \textbf{149} (2017), 1--20.


\bibitem{CFG2} Crismale V., Fidaleo F., Griseta M.E. {\it Spreadability for quantum stochastic processes, with an application to boolean commutation relations}, Entropy \textbf{22} no.5 (2020) Article number 532, 17 pp.







\bibitem{dF} de Finetti B. {\it Funzione caratteristica di un fenomeno aleatorio},  Atti Accad. Naz. Lincei, VI Ser., Mem.
Cl. Sci. Fis. Mat. Nat. \textbf{4} (1931), 251--259.


\bibitem{HS} Hewitt E., Savage L.F.  {\it Symmetric measures on Cartesian products} Trans. Am. Math. Soc. \textbf{80} (1955), 470--501.

\bibitem{Lind} Lindenstrauss J., Olsen G. Sterfeld Y., {\it The Poulsen simplex}, Ann. Inst. Fourier (Grenoble) \textbf{28} (1978), 91--114.

\bibitem{Part} Parthasarathy K. R. {\it On the category of ergodic measures} Illinois J. Math. \textbf{5} (1961), 648--656.
\bibitem{P} Poulsen E. T. {\it A simplex with dense extreme points},  Ann. Inst. Fourier (Grenoble) \textbf{11} (1961), 83--87.

\bibitem{R} Ryll-Nardzewski V.
{\it On stationary sequences of random variables and the de Finetti’s equivalence}, Colloq. Math. \textbf{4} (1957), 149–156.

\bibitem{S} Sakai S. \emph{$C^*$-algebras and $W^*$-algebras}, Springer, Berlin 1971.


\bibitem{Sto} St\o rmer E. {\it Symmetric States of infinite Tensor Products of C*-algebras.}
J. Funct. Anal. \textbf{3} (1969), 48--68.


\bibitem{V} Vincenzi E. {\it $C^*$-dynamical systems and Ergodic Theory
Quantum decoherence, twisted tensor products and De Finetti
theorem}, Ph.D. thesis, https://www.mat.uniroma2.it/dottorato/Theses/2024/VincenziElia.pdf .



\end{thebibliography}
\end{document}